\documentclass[11pt,a4paper]{article}
 \usepackage{euscript}
\usepackage{amsmath}
\usepackage{graphicx}
\usepackage{amsthm}
\usepackage{amssymb}
\usepackage{epsfig}
\usepackage{url}
\usepackage{colonequals}
\usepackage{amsmath,amscd}
\theoremstyle{theorem}
\newtheorem{theorem}{Theorem}[section]
\newtheorem{corollary}[theorem]{Corollary}

\newtheorem{lemma}[theorem]{Lemma}
\newtheorem{proposition}[theorem]{Proposition}

\newtheorem{definition}{Definition}[section]

\newtheorem{example}{Example}[section]
\newtheorem{remark}{Remark}[section]
\numberwithin{equation}{section}
 
\newcommand{\s}{\sigma}
\newcommand{\R}{\text{Re}}

\DeclareSymbolFont{symbolsC}{U}{txsyc}{m}{n}
\DeclareMathSymbol{\notniFromTxfonts}{\mathrel}{symbolsC}{61}
\title{Skew analysis over quaternions. I} 
\author{Masood Aryapoor\\
\tiny{\textit{Division of Mathematics and Physics}}\\
\tiny{\textit{M\"{a}lardalen  University}}\\
\tiny{\textit{Hamngatan 15, 632 17, Eskilstuna, 
Sweden
}}

}
 \date{}
\usepackage{microtype} 
\begin{document}
 \maketitle
\begin{abstract}
\noindent
We introduce a class of rings using which we define the concept of skew regularity for quaternion-valued functions over quaternions. It is shown that the notion of skew regularity coincides with the concept of slice regularity over symmetric slice domains. Known results regarding slice-regular functions over symmetric slice domains are generalized to skew-regular functions over general symmetric domains. Furthermore, we present new results concerning slice-regular functions.     

\end{abstract}
\begin{section}{Introduction} 
This paper consists of two parts: An algebraic part and an analytical part. In the algebraic part, we introduce and study a class of rings which we shall call \textit{skew-convex function rings}.  They can be regarded as generalizations of skew polynomial rings over skew fields. In the analytical part, we use skew-convex function rings over the skew field of quaternions to develop an analytical theory over quaternions which will be called \textit{skew analysis over quaternions}. It turns that skew analysis over quaternions is closely related to the theory of slice-regular functions as introduced in \cite{Anewtheoryregularfunctionsquaternionicvariable2007} and later developed in a series of papers. For a detailed account of slice-regular functions, see \cite{Regularfunctionsofaquaternionicvariable2013} and references therein. 

In Section 2, we introduce the notion of the skew product which is a binary operation on certain sets of functions with values in a fixed skew field (see Definition \ref{(D)product}). The concept of the skew product is motivated by the product formula for the evaluation of skew polynomials introduced in \cite[Theorem 2.7]{LamLeory1988}. We show that the skew product gives rise to near-ring structures on sets of functions with values in the skew field. Restricting our attention to the class of ``skew-convex" functions, we arrive at the notion of skew-convex function rings (see Definition \ref{(D)skew convex} and Theorem \ref{(T)Thetwistedring}). In the rest of Section 2, we study invertible elements in skew-convex function rings. A more complete treatment of the algebraic properties of skew-convex function rings will appear elsewhere.  
 
Section 3 deals with skew analysis over quaternions. In this section, the notion of skew regularity is introduced (see Definition \ref{(D)skewderivatvie}) using which we develop an analytical theory for quaternion-valued functions over quaternions. In particular, we define the notion of the \textit{skew derivative} which is very similar to the classical definition of derivatives (see Definition \ref{(D)skewderivatvie2}). It turns out that skew-regular functions are always skew-convex, a fact with far-reaching consequences. As mentioned above, the concept of skew regularity is closely related to the notion of slice regularity. More precisely, we show that every skew-regular function is slice regular (see Proposition \ref{(P)skewimpliesslice}). More importantly, we will prove that a function defined on a ``symmetric slice domain" is skew regular if and only if it is slice regular (see Subsection \ref{(S)Skew regularity and slice regularity} and Theorem \ref{(T)sliceequivalentskew}). It is known that the notion of slice regularity is not well-behaved over domains that are not slice domains.  The strength of our approach lies in the fact that the concept of skew-regularity is well-behaved even over domains which are not slice domains. One of our main results is that over a general symmetric domain, a function is skew regular if and only if it is symmetrically analytic (see Theorem \ref{(T)skewregularequivalentspherical}).  Furthermore, our approach simplifies and clarifies some of the results and features of slice-regular functions.

\end{section} 


\begin{section}{Rings of skew maps}
The theory of skew polynomial rings over skew fields is rich and well-developed. For a detailed account of this theory,  see for example \cite[Section 1.1]{Freeidealrings}.  In \cite{LamLeory1988}, the authors defined the evaluation of skew polynomials using which skew polynomials can naturally be considered as maps on the ground skew field.    
It turns out that the value of the product of two skew polynomials at a given point is not necessarily equal to the product of the values of the skew polynomials at the same point. The correct formula, called the product formula, is given in  \cite[Theorem 2.7]{LamLeory1988}. The product formula can be regarded as a binary operation which we shall call the \textit{skew product}. This section will deal with the skew product and its properties. Subsection 2.1 introduces the notions of the skew product and skew-convex functions, and gives their basic properties. In Section 2.2, we study skew-invertible functions, that is, functions which are invertible with respect to the skew product.    

\begin{subsection}{Rings of skew-convex functions}
	Let $K$ be  a skew field and $X$ be a set on which the multiplicative group $K^*\colonequals K\setminus \{0\}$ acts (on the left). The action of $a\in K^*$ on $x\in X$ is denoted by $ ^{a}x$. We will freely use the standard concepts from Group Theory. In particular, we use the following notions: An \textit{invariant} subset of $X$ is a subset $Y$ of $X$ which satisfies the property that $a\in K^*$ and $y\in Y$ imply that $^{a}y\in Y$; By an \textit{orbit}, we mean the smallest invariant subset containing some element. 
	
	We denote the set of all functions $f\colon  X\to K$ by $\mathcal{F}(X)$. By abuse of notation, a constant function in $ \mathcal{F}(X)$, whose value is $a\in K$,  is simply denoted by $a$.
	Given functions $f,g\colon  X\to K$,  let $f+g$ denote the pointwise addition of the functions $f$ and $g$. 
	\begin{definition}
		The \textit{left skew product} of the functions $f$ and $g$ as follows
		\begin{equation}\label{(D)product}
			(f\diamond g)(a)=\begin{cases}
				f\left(  \, ^{{g(a)}}a\, \right) g(a) & g(a)\neq 0,\\
				0& g(a)=0.
			\end{cases}
		\end{equation}
	\end{definition}
	The \textit{right skew product} is defined as follows
	\begin{equation}\label{(D)rightproduct}
		(f\diamond_r g)(a)=\begin{cases}
			f(a)g\left(  \, ^{{f(a)^{-1}}}a\, \right)  & f(a)\neq 0,\\
			0& f(a)=0.
		\end{cases}
	\end{equation}

	In this paper, we will mostly work with the left skew product.  Therefore, we usually drop the adjective ``left" if there is no risk of confusion.  We leave it to the reader to formulate and prove  similar  results for the right skew product.
	 
	It is easy to see that $(a\diamond f)(x)=af(x)$, for every $a\in K,x\in X$. We shall henceforth denote $a\diamond f$ by $af$. In the following lemma, the proof of which is straightforward, we collect some properties of the skew product. 
	\begin{lemma}\label{(L)propertiesofsproduct}
		Let $f,h,g\colon  X\to K$ be arbitrary functions. Then:\\
		(1) The constant function $1\colon  x\mapsto 1$ is a unit for $\diamond$, that is, $f=f\diamond 1=1\diamond f$.\\
		(2) $(f+g)\diamond h=  f\diamond h+ g\diamond h$, that is, the right distributive law (with respect to pointwise addition) holds for the skew product.  \\
		(3) $ (f\diamond g) \diamond h = f \diamond (g\diamond h)$, that is, $\diamond$ is associative. 
	\end{lemma}
	It follows from this lemma that the set $\mathcal{F}(X)$ equipped with pointwise addition and the skew product is a structure known as right near-ring in the literature. 
	We note  that the skew product may not be left distributive with respect to pointwise addition. However, the left distributive law for $\diamond$ holds for a class of functions  described below. 
	\begin{definition}\label{(D)skew convex}
		A function $f\colon  X\to K$ is called skew convex if 
		$$f\diamond(a+b)=f\diamond a+f\diamond b, \text{ for all } a,b\in K.$$
		The set of all skew-convex functions $f\colon  X\to K$ is denoted by  $\mathcal{S}(X)$. 
	\end{definition}	
	Any constant function belongs to $\mathcal{S}(X)$ since $a\diamond b=ab$ for all $a,b\in K$. More generally, we have the following result. The proof is left to the reader.
	\begin{proposition}\label{(P)constantskewfunction}
		A function $f\colon  X\to K$ which is constant on every orbit in $X$ is skew convex. 
	\end{proposition}
	The following lemma justifies the introduction of $\mathcal{S}(X)$.
	\begin{lemma}\label{(L)leftdist}
		Let $h\colon  X\to K$ be given. The condition 	$$ h\diamond(f+g) = h\diamond f + h\diamond g,$$
		holds for all functions $f,g\colon  X\to K$ if and only if $h$ is skew convex.
	\end{lemma}
	\begin{proof}
	The result follows from the identity 
	$$(h\diamond f)(x)=(h\diamond {f(x)})(x), \text{ for all } h,f\in \mathcal{F}(X) \text{ and } x\in X.$$
	\end{proof}	
	As a consequence of this lemma, we have the following result.  
	\begin{theorem}\label{(T)Thetwistedring}
	Equipped with the left skew product, the additive group $\mathcal{S}(X)$  is a ring with identity. 
	\end{theorem}
	\begin{proof}
		The result follows from Lemma \ref{(L)leftdist} and the  general fact that in any right near-ring $R$, the set $$\{r\in R\,|\, r(s_1+s_2)=rs_1+rs_2\text{ for all } s_1,s_2\in R\},$$
		is a ring. 
	\end{proof}
	We call $\mathcal{S}(X)$, equipped with pointwise addition and the skew product,  \textit{the ring of skew-convex functions} on $X$ determined by the action of $K^*$ on $X$. To introduce a subring of $\mathcal{S}(X)$, let us denote the set of all functions $f\colon  X\to K$ which are constants on every orbit, by $\mathcal{C}(X)$. By Proposition \ref{(P)constantskewfunction}, we have $\mathcal{C}(X)\subset \mathcal{S}(X)$. Moreover, we have the following proposition whose proof is left to the reader. 
	\begin{proposition}\label{(P)Constantskewconvex}
		The set $\mathcal{C}(X)$ is a subring of $\mathcal{S}(X)$. Moreover, it is isomorphic to the ordinary ring of functions $f\colon  X/K^*\to K$ (equipped with pointwise addition and pointwise multiplication). Here, $X/K^*$ denotes the set of equivalence classes of the action. 
	\end{proposition}

	Let us now give some examples. 
	\begin{example}
		If the action of $K^*$ on $X$ is trivial, the ring $\mathcal{S}(X)$ is just the ring of all functions $f\colon  X\to K$ equipped with pointwise addition and pointwise multiplication. 
	\end{example}
	\begin{example}\label{(Ex)(s,delta)-skewmaps}
		Let $\s\colon  K\to K$ be an endomorphism and $\delta\colon  K\to K$ be a $\s$-derivation. Let $K[T;\s,\delta]$ denote the ring of skew polynomials determined by $\s$  and $\delta$. Every elements of $K[T;\s,\delta]$ can uniquely be written as $\sum_{m=0}^na_mT^m$ where $a_m\in K$ with $a_n\neq 0$. Recall that the identity $Ta=\s(a)T+\delta(a)$, where $a\in K$, holds in $K[T;\s,\delta]$. One can define  the $(\s,\delta)$-action of $K^*$ on $K$ as follows
		\begin{equation}\label{(E)conjugate}
			^{b}a=\s(b)ab^{-1}+\delta(b)b^{-1}.
		\end{equation} 
		The ring of skew-convex functions determined by this action is denoted by $K[\s,\delta]$. One can verify that there exists a unique ring homomorphism $$K[T;\s,\delta]\to K[\s,\delta],$$ which sends each $a\in K$ to the constant function $a$, and $T$ to the identity function $1_K\colon  K\to K$. In particular,  any skew polynomial $P(T)\in K[T;\s,\delta]$ can, under this homomorphism, be considered as a skew-convex function $P\colon  K\to K$. The reader can  verify that the value $P(a)$ of $P$  at $a\in K$ coincides with the evaluation map introduced in \cite{LamLeory1988}, that is, $P(a)$ is the unique element of $K$ for which we have
		$$P(T)-P(a)\in K[T;\s,\delta](T-a).$$   
		In the particular case where $\delta$ is zero and $\s$ is the identity homomorphism, the value $P(a)$ is computed as follows
		$$P(a)=\sum_{m=0}^na_ma^m, \text{ where } P(T)=\sum_{m=0}^na_mT^m\in K[T]\colonequals K[T;1_K,0].$$
	\end{example}
	We conclude this part with a discussion regarding pullbacks. 
	Let  $Y$ be a second set on which $K^*$ acts (on the left). For a map $\phi\colon  X\to Y$, let  $\phi^*\colon  \mathcal{F}(Y)\to \mathcal{F}(X)$ denote the pullback map defined by $\phi^*(f)=f\circ \phi$.  Recall that a map $\phi\colon  X\to Y$ is called action-preserving if  $\phi(^{a}x)={}^{a}\phi(x)$ for all $a\in K^*$ and $x\in X$.
	\begin{proposition}\label{(P)actionpreserving}
		Suppose that $\phi\colon  X\to Y$ is action-preserving. Then, for any function $f\in \mathcal{S}(Y)$, we have $\phi^*(f)\in \mathcal{S}(X)$. Moreover, the map 
		$\phi^*\colon  \mathcal{S}(Y)\to \mathcal{S}(X)$ is a homomorphism of rings. 
	\end{proposition}
\begin{proof}
	The fact that $\phi^*(f)\in \mathcal{S}(X)$ for $f\in \mathcal{S}(Y)$, follows from the identity $$(f\circ \phi)\diamond a=(f\diamond a)\circ \phi, \text{ where } a\in K.$$ 
	The rest of the proof is straightforward. 
\end{proof}
	In particular, if $Z\subset X$ is an invariant subset of $X$, then the restriction map $f\in \mathcal{S}(X)\mapsto f |_Z\in \mathcal{S}(Z)$ is a well-defined homomorphism of rings.  
	
\end{subsection}
\begin{subsection}{Skew-invertible  functions}

	As before, let $K$ be a skew field and $X$ be a set on which $K^*$ acts. A function $f\in \mathcal{F}(X)$ is called \textit{skew invertible} if it is invertible with respect to the skew product, in which case, the inverse of $f$ is called its \textit{skew inverse} and denoted by $f^{\langle-1\rangle}$. The notion of skew inverse with respect to the right skew product is defined in the obvious way. The skew inverse of a function $f$ with respect to the right skew  product will be denoted by $f^{\langle-1\rangle_r}$.
	\begin{lemma}\label{(L)Inverseofc-invariant}
		 If $f\in \mathcal{S}(X)$ is skew invertible, then the skew inverse of $f$ belongs to $\mathcal{S}(X)$.   
	\end{lemma}
	\begin{proof}
		By Lemma \ref{(L)leftdist}, we have 
		$$\mathcal{S}(X)=\{f\in \mathcal{F}(X)\,|\, f\diamond(g+h)=f\diamond g+f\diamond h\text{ for all } g,h\in \mathcal{F}(X)\}.$$
		The result follows from the following general fact whose proof is left to the reader: Let $R$ be a right near-ring and consider the ring 
		$$R'=\{r\in R\,|\, r(s_1+s_2)=rs_1+rs_2\text{ for all } s_1,s_2\in R\}.$$
		If $r\in R'$ is invertible in $R$, then its inverse belongs to $R'$. 
	\end{proof}
	   Next, we give a characterization of skew-invertible functions.
	\begin{lemma}\label{(L)Invertibility} 
		(1)  Let $f\in \mathcal{F}(X)$. There exists $g\in \mathcal{F}(X)$ such that $f\diamond g=1$ if and only if for every $x\in X$, there exists some $a\in K^*$ such that $f({}^ax)=a^{-1}$.\\
		(2) Let $g\in \mathcal{F}(X)$.  There exists $f\in \mathcal{F}(X)$ such that $f\diamond g=1$ if and only if $g(X)\subset K^*$, and the map $x\mapsto {}^{g(x)}x$ is 1-1. \\
		(3) A function $f\in \mathcal{F}(X)$ is skew invertible if and only if $f(X)\subset K^*$ and the assignment $x\mapsto {}^{f(x)}x$ establishes a bijection from $X$ onto $X$. 
	\end{lemma}
	\begin{proof}
		(1)  The proof is straightforward.\\
		(2)   First, we prove the ``if" direction. It follows from $f\diamond g=1$ that $f({}^{g(x)}x)g(x)=1$ for all $x\in X$. Clearly, $g$ must be nonzero on $X$, i.e., $g(X)\subset K^*$.  If ${}^{g(x)}x={}^{g(y)}y=z\in K$ for some $x,y\in X$, then we have $f(z)g(x)=f(z)g(y)=1$, implying that $g(x)=g(y)$. One can easily verify that ${}^{g(x)}x={}^{g(y)}y$ and $g(x)=g(y)$ imply that $x=y$. This proves the ``if" direction. Conversely, let $g$ satisfy the stated properties. We define a map $f\colon  X\to K$ as follows: If $y={}^{g(x)}x$ for some $x\in X$, we set $f(y)=g(x)^{-1}$. Otherwise, we set $f(y)=x_0$ where $x_0\in X$ is a fixed element. It is easy to see that $f$ is well-defined and  $f\diamond g=1$. \\
		(3) Suppose that $f$ is skew invertible and let $g$ be its skew inverse. By $(2)$,  $f$ is nonzero on $X$ and $x\mapsto {}^{f(x)}x$ is 1-1. The fact that $x\mapsto {}^{f(x)}x$ is onto follows from the following identity
		$${}^{f(^{g(y)}y)}\left( {}^{g(y)}y\right)=y.$$ 
		Conversely, suppose that $f$ satisfies the stated properties. By Part $(2)$, there exists $g\in \mathcal{F}(X)$ such that $g\diamond f=1$. We need only  show that $f\diamond g=1$. Given an arbitrary element $x\in X$, we can choose $y\in X$ such that $x={}^{f(y)}y$. We have
		$$(g\diamond f)(y)=1\implies g({}^{f(y)}y)f(y)=1\implies g(x)f(y)=1.$$
		Therefore, we have ${}^{g(x)}x={}^{g(x)f(y)}y=y$ from which it follows that 
		$$g(x)f(y)=1\implies f(y)g(x)=1\implies f({}^{g(x)}x)g(x)=1 \implies (f\diamond g)(x)=1.$$
		Since $x\in X$ was arbitrary, we conclude that $f\diamond g=1$.
	\end{proof}
	A skew-invertible function shares some properties with its skew inverse. In the following proposition, we give two such properties.
	\begin{proposition}\label{(P)injectivsurjective}
		Let $f\colon  X\to K^*$ be a skew-invertible function. Then:\\
		(1) If $f$ is 1-1, so is $f^{\langle-1\rangle}$.\\
		(2) If $f$ is onto, so is $f^{\langle-1\rangle}$.
	\end{proposition}
	\begin{proof}
		By Part (3) of Lemma \ref{(L)Invertibility}, the function $\alpha\colon  X\to X$, defined by $\alpha(x)={}^{f(x)}x$, is a bijection. Let $\beta\colon  X\to X$ denote its inverse (with respect to function composition). It is easy to check that 
		$$f^{\langle-1\rangle}(x)=f(\beta(x))^{-1}, \text{ for all } x\in X,$$
		from which the results follow. 
	\end{proof}
	Regarding skew-invertible elements of $\mathcal{S}(X)$, we have the following characterization. 
\begin{proposition}\label{(P)inverserightinverse}
	A skew-convex function $f:X\to K$ satisfying $f(X)\subset K^*$ is skew invertible if and only if for any $x\in X$, there exists some $a\in K^*$ such that $f({}^ax)=a^{-1}$.  
\end{proposition}

\begin{proof}
	The ``only if" direction follows from Part 1 of Lemma \ref{(L)Invertibility}. To prove the converse, we use  Part (2) of Lemma \ref{(L)Invertibility}. Therefore, we need only show that if ${}^{f(x)}x={}^{f(y)}y$, then $x=y$. Let ${}^{f(x)}x={}^{f(y)}y$ for some $x,y\in X$. Then $y$ and $x$ are in the same orbit, implying that there exists $a\in K^*$ such that $y={}^{a}x$. So, 
	$${}^{f(x)}x={}^{f(y)}y={}^{f({}^{a}x)a}x.$$
	It follows that $f(x)b=f({}^ax)a$ for some $b\in K^*$ satisfying ${}^bx=x$. Note that $f({}^{-a}x)=f({}^ax)$ because 
	$$0=f\diamond (a+(-a))=f\diamond a +f\diamond (-a)\implies f\diamond (-a)=-f\diamond a.$$ 
	Therefore, we can write
	$$0=f({}^bx)b-f({}^{-a}x)a=(f\diamond b+f\diamond (-a))(x)=\left( f\diamond (b-a)\right)(x),$$
	implying that $a=b$. Since $f(X)\subset K^*$, we must have $x={}^bx={}^ax=y$.
\end{proof}
	     We conclude the algebraic part of the paper with an example of a skew-invertible function. 
	\begin{example}\label{(Ex)Invertibility}
		As in Example \ref{(Ex)(s,delta)-skewmaps}, let $\s\colon  K\to K$ be an endomorphism and $\delta\colon  K\to K$ be a $\s$-derivation. Let $A\subset K$ be an invariant set under the $(\s,\delta)$-action. Using Lemma \ref{(L)Invertibility} and Proposition \ref{(P)inverserightinverse}, we see that a linear  polynomial $T-c\in K[T;\s,\delta]$, where $c\in K$, is skew invertible as an element of $\mathcal{S}(A)$ if and only if (1) $T-c$ does not have a root in $A$, and (2) for any $a\in A$, there exists $b\in K^*$ such that $(T-c)({}^ba)=b^{-1}$. The first condition is equivalent to $c\notin A$. The second condition is equivalent to the condition that for any $a\in A$, there exists $b\in K^*$ satisfying the equation 
		${}^ba-c=b^{-1}$, or equivalently, the equation  $$\s(b)a+\delta(b)-cb=1.$$
		Note that if $T-c$ is skew invertible, then for  any $a\in A$, we have
		$$ b=(T-c)^{\langle-1\rangle}(a)\iff \s(b)a+\delta(b)-cb=1.$$
		In the case where $\delta$ is zero and $\s$ is the identity homomorphism, we have 
		$$ b=(T-c)^{\langle-1\rangle}(a)\iff ba-cb=1.$$
	\end{example}

\end{subsection}




\end{section} 
\begin{section}{Skew analysis over quaternions}
	A central problem of ``quaternionic analysis" is to find a suitable notion of ``quaternionic derivative''. It is known that the most obvious definition, namely, 
	   $$F'(q_0)=\lim_{q\to q_0}\left( F(q)-F(q_0)\right)(q-q_0)^{-1},$$
	is too restrictive. From the point of view of the skew product, one is led to the informal definition 
	 $$F'(q_0)=\lim_{q\to q_0}\left( \left( F-F(q_0)\right)\diamond (T-q_0)^{\langle -1 \rangle}\right)(q),$$
	 which we shall call the \textit{skew derivative} (see Definition \ref{(D)skewderivatvie2} for the precise definition). This section deals with the notion of the skew derivative and develops a version of quaternionic analysis based on the skew derivative.   
	
	Throughout this section,  $\mathbb{H}$ denotes the skew field of quaternions. Given a quaternion 
	$$q=r_0+r_1i+r_2j+r_3k, \text{ where } r_0,r_1,r_2,r_3\in\mathbb{R},$$
	we denote the real part, imaginary part, conjugate and norm of $q$ by 
	$$\text{Re}(q)=r_0,\, \text{Im}(q)=r_1i+r_2j+r_3k,$$ $$\overline{q}=r_0-r_1i-r_2j-r_3k  \text{ and } |q|=\sqrt{r_0^2+r_1^2+r_2^2+r_3^2}, \,\, \text{ respectively.}$$ 
	In this section, we fix the following action of  $\mathbb{H}^*$ on $\mathbb{H}$:
	${}^{p}q=pqp^{-1}.$ The orbit containing a quaternion $q\in\mathbb{H}$ is denoted by $O(q)$. It is known that for all $p,q\in \mathbb{H}$, $O(p)=O(q)$ holds if and only if $\text{Re}(p)=\text{Re}(q)$ and $|p|=|q|$ hold. In particular, an orbit  is either a single point, called a \textit{trivial orbit}, or a 2-dimensional sphere, called a  \textit{nontrivial orbit}. We also consider $\mathbb{H}$ as a topological space where the topology is  the Euclidean topology induced by the norm $q\mapsto |q|$.  For convenience, we shall use the following terminology: A subset of $\mathbb{H}$ is called a \textit{region} if it is both open and invariant under the action ${}^{p}q=pqp^{-1}.$ 
	 
	 \begin{subsection}{Skew-convex functions on $\mathbb{H}$}\label{(S)SkewConvex}	
	 	In this section, we study skew-convex functions on invariant subsets of $\mathbb{H}$. We begin with a characterization of skew-convex functions on orbits. 
	 	\begin{proposition}\label{(P)CharacterzationSkewconvex}
	 		Let $F\colon  O\to \mathbb{H}$ be a function on an orbit $O$ of $\mathbb{H}$. Then, $F$ is skew convex if and only if $F$ is (left) affine in the sense that there exist constants $a,b\in \mathbb{H}$ such that
	 		$$F(q)=a+bq, \text{ for all } q\in O.$$
	 	\end{proposition}
	 	\begin{proof}
	 		If $O$ is a trivial orbit, there is nothing to prove, so assume that $O$ is nontrivial. The proof of the ``if" direction is straightforward. To prove the converse, fix $q_0\in O$. We need only show that
	 		$$(F(q_1)-F(q_0))(q_1-q_0)^{-1}=(F(q_2)-F(q_0))(q_2-q_0)^{-1}, $$
	 		for all $q_1\neq q_2\in O\setminus \{q_0\}$. This identity can be rewritten as
	 		\begin{equation}\label{(I)rewritten}
	 		F(q_1)(q_1-q_0)^{-1}-F(q_2)(q_2-q_0)^{-1}=F(q_0)\left( (q_1-q_0)^{-1}-(q_2-q_0)^{-1}\right).	
	 		\end{equation}
	 		Since $q_i\in O(q_0)$, we have
	 		$$q_i=(q_i-q_0)^{-1}\overline{q}_0(q_i-q_0), \text{ for } i=1,2,$$
	 		using which the left-hand side of Identity \ref{(I)rewritten} can be rewritten as 
	 		$$
	 		F({}^{(q_1-q_0)^{-1}}\overline{q}_0)(q_1-q_0)^{-1}-F({}^{(q_2-q_0)^{-1}}\overline{q}_0)(q_2-q_0)^{-1}
	 		$$
	 		$$
	 		=F( {}^{\left( (q_1-q_0)^{-1}-(q_2-q_0)^{-1}\right) }\overline{q}_0)\left( (q_1-q_0)^{-1}-(q_2-q_0)^{-1}\right).
	 		$$
	 		It is easy to check that ${}^{\left( (q_1-q_0)^{-1}-(q_2-q_0)^{-1}\right) }\overline{q}_0=q_0$ form which Identity \ref{(I)rewritten} follows. 
	 	\end{proof}
 	In view of this proposition, the following definition is well-defined.
 	\begin{definition}\label{(D)OrbitalDerivatvie}
 		The orbital derivative of a skew-convex function $F:O\to\mathbb{H}$ at a nontrivial orbit $O$ is defined to be the value 
 		$$\left(F(p)-F(q) \right)(p-q)^{-1}, \text{ where } q\neq p\in O.$$
 		The orbital derivative of $F$ at $O=O(q)$ is denoted by $F^{o}(O)$ or $F^o(q)$. If $O$ is a trivial orbit, the orbital derivative of $F$ at $O$ is defined to be zero. 
 	\end{definition}
	 Let us give another characterization of skew-convex functions which will be needed later. In the following proposition, the function $T\colon  \mathbb{H}\to \mathbb{H}$ denotes the identity function on $\mathbb{H}$. 
 		\begin{proposition}\label{(P)SkeconvexT-q}
 			Let $O\subset\mathbb{H}$ be a nontrivial orbit. A function $F\colon  O\to \mathbb{H}$ is skew convex on $O$ if and only if $F$ can be written as $$F=a+E\diamond (T-q),$$ for some $a\in \mathbb{H}$, $q\in O$ and some function $E\colon  O\to \mathbb{H}$, in which case,  $E(\overline{q})$ is equal to the orbital derivative of $F$ at $O$. 
 		\end{proposition}
 		\begin{proof}
 			The ``only if" direction follows directly from Proposition \ref{(P)CharacterzationSkewconvex}. To prove the other direction,  let  $F=a+E\diamond (T-q)$, for some $a\in \mathbb{H}$, $q\in U$ and some function $E\colon  O\to \mathbb{H}$. Evaluating at a quaternion ${}^{p}q$, where $p\in\mathbb{H}$ satisfies $pq\neq qp$, we have
 			$$F({}^{p}q)=a+E({}^{({}^pq-q)p}q)({}^{p}q-q)=a+E({}^{pq-qp}q)({}^pq-q).$$
 			Using the elementary formula 
 			$${}^{pq-qp}q=\overline{q}, \text{ for all } p,q\in\mathbb{H} \text{ with } pq\neq qp,$$
 			we obtain $F({}^{p}q)=a+E(\overline{q})({}^pq-q)$, or equivalently, 
 				$$F(p)=a+E(\overline{q})(p-q), \text{ for all } p\in O.$$ 	
 			In particular, $F$ is skew convex. Note that $a=F(q)$ and $$E(\overline{q})=\left(F(p)-F(q) \right)(p-q)^{-1}=F^{o}(O), \text{ where } q\neq p\in O.$$ 
 		\end{proof} 
 	 In the following proposition, we collect some properties of the orbital derivative. 
 	\begin{proposition}\label{(P)PropertiesOrbitalD}
 		Let $F, G\colon  U\to \mathbb{H}$ be skew-convex functions on a region $U$ in $\mathbb{H}$. Then, $F+G$ and $F\diamond G$ are skew convex on $U$, and moreover, we have 
 		$$(F+G)^{o}=F^{o}+G^{o}, \text{ on $U$  and },$$
 		$$(F\diamond G)^{o}(p)=(F^{o}\diamond G)(p)+(F\diamond G^{o})(\overline{p}), \text{ for all $p\in U$ }.$$
 	\end{proposition}
 	\begin{proof}
 		The first part follows from Theorem \ref{(T)Thetwistedring}. The first formula is trivial. To prove the second formula, let $p\in U$ be fixed and set $O=O(p)$. We assume that $G(p)\neq 0$. The case $G(p)=0$ is left to the reader.  Then, we have
 		$$G=G(p)+G^{o}(O)(T-p), \text{ as functions on } O,$$
 		$$F=F({}^{G(p)}p)+F^{o}(O)(T-{}^{G(p)}p)\text{ as functions on } O.$$
 		It easy to check that the following identity holds on $U$: 
 		$$F\diamond G=F({}^{G(p)}p)G(p)+\left(F^{o}(O)G(p)+F\diamond (G^{o}(O))  \right)\diamond (T-p).$$
 		By Proposition \ref{(P)SkeconvexT-q}, we have
 		\begin{align*}
 		(F\diamond G)^{o}(p)=& \left(F^{o}(O)G(p)+F\diamond (G^{o}(O) ) \right)(\overline{p})\\
 		=&(F^{o}\diamond G)(p)+(F\diamond G^{o})(\overline{p}).
 		\end{align*}
 	\end{proof}	
	 	Now, we give some examples of skew-convex functions on $\mathbb{H}$. The functions $q\mapsto \R(q)^n$ and $q\mapsto |q|^n$, where $n$ is a natural number,  are skew convex on $\mathbb{H}$ since they are constant on all orbits. As another examples of skew-convex functions, we have polynomials $P(T)\in\mathbb{H}[T]$, where $\mathbb{H}[T]$ is the ring of polynomials in a central indeterminate $T$ over $\mathbb{H}$ (see Example \ref{(Ex)(s,delta)-skewmaps}). \text{Re}call that a (left) polynomial  $P(T)=\sum_{m=0}^nq_mT^m\in\mathbb{H}[T]$, as a function on $\mathbb{H}$, is evaluated as follows  $$P(q)=\sum_{m=0}^nq_mq^m.$$ 
	 	It is easy to verify that the orbital derivative of  $P(T)=\sum_{m=0}^nq_mT^m\in\mathbb{H}[T]$ at an orbit $O(q)$ equals
	 	$$P^o(q)=\sum_{m=1}^nq_m\left( \sum_{k=0}^{m-1} \overline{q}^kq^{m-1-k}\right) .$$
	 	To introduce the next class of skew-convex functions on $\mathbb{H}$, we need the following result.  	 
 			\begin{proposition}\label{(P)limitofskewfunctions}
 				Let $X\subset \mathbb{H}$ be invariant and $F_1,F_2,...$ be a sequence of skew-convex $\mathbb{H}$-valued functions on $X$. Suppose that the sequence $F_1,F_2,...$ converges pointwise to a function $F\colon  X\to \mathbb{H}$. Then, $F$ is skew convex on $X$. 
 			\end{proposition}
 			\begin{proof}
 				One can easily prove the proposition using the definition of a skew-convex function (see Definition \ref{(D)skew convex}). 
 			\end{proof}
 		As an application of this proposition, we show that (left) spherical series are skew convex. Let us review some facts regarding spherical series over $\mathbb{H}$ (see also \cite{Anewseriesexpansion}). Suppose that $O\subset \mathbb{H}$ is an orbit. We associate a polynomial $P_O(T)\in \mathbb{H}[T]$ to $O$ as follows: If $O=\{r\}$ is a trivial orbit, we set $P_O(T)=T-r$. If $O$ is a nontrivial orbit, we put
 		$$P_O(T)=T^2-2\R(p)T+|p|^2, \text{ where } p\in O.$$
 		Note that $P_O(T)$ does not depend on the choice $p\in O$. By a \textit{(formal) spherical series centered at $O$}, we shall mean a series of the form
 		$$S(T)=\sum_{n=0}^\infty q_nP_O(T)^n,$$
 		where $q_0,q_1,...\in \mathbb{H}$. The value $S(q)$ at a point $q\in\mathbb{H}$ is defined to be 
 		$$S(q)=\sum_{n=0}^\infty q_nP_O(q)^n,$$ 
 		if the series converges. The notion of right spherical series is defined in a similar fashion. We note that a spherical series centered at a trivial orbit $O=\{r\}$ is just an ordinary series of the form 
 		$$\sum_{n=0}^\infty q_n(T-r)^n.$$
 		The following proposition is easily proved using the Weierstrass $M$-test.   
 		\begin{proposition}\label{(P)sphericalpower}
 		Let $\sum_{n=0}^\infty q_nP_O(T)^n$ be a  spherical power series centered at an orbit $O$. Suppose, furthermore, that $\limsup_{n\to \infty} |q_n|^{1/n}=1/R$ for some $R>0$. Then, the series converges absolutely and uniformly on the compact subsets of the region 
 			$$U(O,R) \colonequals \{q\in \mathbb{H}\,|\, |P_O(q)|<R \}.$$
 		\end{proposition}
 		The number $R$ in this Proposition will be called the \textit{radius of convergence} of the spherical power series. The region  $U(O,R)$ will be called the \textit{region of convergence} of the spherical power series.
 		It follows immediately from Proposition \ref{(P)limitofskewfunctions} that any spherical series which is convergent on an invariant set $X$ is skew convex on $X$. It is clear that the orbital derivative of a spherical series is zero at its center. 
 		
 		We conclude this part with the following lemma the proof of which is left to the reader. 
 	\begin{lemma}\label{(L)limitskewconvexorbit}
 		Let $F\colon  U\to\mathbb{H}$ be a function on a region and $O\subset U$ be an orbit. If $F$ is skew convex on $U\setminus O$ and is continuous at all $p\in O$, then $F$ is skew convex on $U$. 
 	\end{lemma}
 		
	 \end{subsection}
 \begin{subsection}{Skew-invertible functions on $\mathbb{H}$}\label{(S)Skewinvertible}	
 	In this section, we study the question of skew invertibility for skew-convex functions on invariant subsets of $\mathbb{H}$. We begin with the following result.	
 	\begin{proposition}\label{(P)skeweinvertibleonH}
 		Let $F\colon  O\to \mathbb{H}$ be a skew-convex function on an orbit $O\subset \mathbb{H}$. Then, $F$ is  skew invertible if and only if  $F(x)=0$ has no solutions on $O$.
 	\end{proposition}
 	\begin{proof}
 		The result is clear if $O$ is a trivial orbit.  So, we assume that $O$ is nontrivial. The ``only if" direction follows from Part 3 of Lemma \ref{(L)Invertibility}. We prove the other direction.  
 		Since $F$ is skew-convex, there exists constants $a,b\in \mathbb{H}$ such that $F(p)=a+bp$ for all $p\in O$. 
 		By Lemma \ref{(L)Invertibility} and Proposition \ref{(P)inverserightinverse}, we need only show that for any $p\in O$, there exists $x\in \mathbb{H}^*$ such that $F({}^{x}p)=x^{-1}$. For a fixed $p\in O$, we need to show that the equation
 		$$a+b{}^{x}p=x^{-1} \text{ or } ax+bxp=1,$$
 		has a nonzero solution. The assignment $x\mapsto ax+bxp$ defines an $\mathbb{R}$-linear map on $\mathbb{H}$. Since its kernel is zero, it must be onto, and therefore, there exists $x\in\mathbb{H}^*$ such that $ax+bxp=1$.
 	\end{proof}
 	As an example, we see that  $T-q\in \mathbb{H}[T]$ is skew invertible on  $\mathbb{H}\setminus O(q)$ (see also Example \ref{(Ex)Invertibility}). Moreover, we have the following description of  $(T-q)^{\langle-1\rangle}$.
 	\begin{proposition}\label{(P)T-c}
 		For any $q_0\in \mathbb{H}$, the function $T-q_0$ is skew invertible on $\mathbb{H}\setminus O(q_0)$ and its skew inverse satisfies
 		$$(T-q_0)^{\langle-1\rangle}(q)=(q-\overline{q}_0)\left( q^2-2\text{Re}(q_0)q+|q_0|^2\right)^{-1}, \text{ for all } q\notin O(q_0).
 		$$
 		In particular,  we have
 		$$(T-q_0)^{\langle-1\rangle}(q)=\left( q-q_0\right)^{-1}, \text{ for all } q\notin O(q_0) \text{ with } qq_0=q_0q.
 		$$
 		Moreover, $(T-q_0)^{\langle-1\rangle}\colon \mathbb{H}\setminus O(q_0)\to \mathbb{H}^*$ is 1-1 and its image is the set 
 		$$\mathbb{H}^*\setminus \{x\in\mathbb{H}\,|\,\R(x)=0 \text{ and } \R(q_0x)=-\frac{1}{2}\}.$$
 	\end{proposition}
 	\begin{proof}
 		It is known that the equation $xq-q_0x=1$ has a (unique) solution $x\in \mathbb{H}$ if  $q\notin O=O(q_0)$ (see \cite{JohnsonOntheEqua}). In fact, the unique solution of $xq-q_0x=1$ is
 		$$x=(q-\overline{q}_0)\left( q^2-2\text{Re}(q_0)q+|q_0|^2\right)^{-1}.$$
 		As seen in Example \ref{(Ex)Invertibility}, it follows that $T-q_0$ is skew invertible on $\mathbb{H}\setminus O$ and its skew inverse $(T-q_0)^{\langle-1\rangle}$ is skew convex. 
 		Moreover, we have 
 		$$(T-q_0)^{\langle-1\rangle}(q)=(q-\overline{q}_0)\left( q^2-2\text{Re}(q_0)q+|q_0|^2\right)^{-1},
 		$$
 		for all $q\in \mathbb{H}\setminus O$. The fact that the skew inverse of $T-q_0$ is 1-1 follows from proposition \ref{(P)injectivsurjective}. To prove the last assertion, we note that the equation $xq-q_0x=1$ gives $q=x^{-1}q_0x+x^{-1}$. It follows from the above discussion that the image of $(T-q_0)^{\langle-1\rangle}$ consists of all $x\in\mathbb{H}^*$ such that $x^{-1}q_0x+x^{-1}\notin O$. It is easy to show that 
 		$$x^{-1}q_0x+x^{-1}\in O\iff \R(x)=0 \text{ and } \R(q_0x)=-\frac{1}{2},$$
 		from which the last part of the proposition follows. 
 	\end{proof}
 \begin{remark}\label{(R)rightskewinverse}
 	Similarly, one can show that $T-q_0$ is skew invertible on $\mathbb{H}\setminus O(q_0)$ with respect to the right skew product. Moreover, the skew inverse  $(T-q_0)^{\langle-1\rangle_r}$ of $T-q_0$ satisfies 
 	$$(T-q_0)^{\langle-1\rangle_r}(q)=\left( q^2-2\text{Re}(q_0)q+|q_0|^2\right)^{-1}(q-\overline{q}_0), \text{ for all } q\notin O(q_0).
 	$$
 	In \cite{ACauchykernel}, the function $(T-q_0)^{\langle-1\rangle_r}$ is studied in more detail. 
 \end{remark}
 	Now, we treat polynomials of higher degree. Let $P(T)\in\mathbb{H}[T]$ be a nonzero polynomial. Set 
 	$$Z(P)=\{q\in\mathbb{H}\,|\, \exists p\in O(q) \text{ such that }P(p)=0\}.$$
 	It follows directly from Proposition \ref{(P)skeweinvertibleonH} that $P$ is skew invertible on the set $\mathbb{H}\setminus Z(P)$. 
 	If $P(T)\in\mathbb{R}[T]$ is a polynomial with real coefficients, then it is easy to see that
 	$$P(T)^{\langle -1 \rangle}(q)=P(q)^{-1}, \text{ for all } q\notin Z(P).$$
 	Therefore, there  is no risk of confusion in using the notation $P(T)^{-1 }$ instead of $P(T)^{\langle -1 \rangle}$ when $P(T)$ has real coefficients.  For a detailed account of polynomials over quaternions, see \cite[Chapter 3]{Powerseriesanalyticityquaternions2012}.

 \end{subsection}
 \begin{subsection}{Limits and the skew product}
 	In this part, we investigate the behavior of the skew product with respect to limits.  We begin with a simple lemma whose proof is left to the reader.
 	\begin{lemma}\label{(L)limitComposition}
 		Let $U\subset \mathbb{H}$ be a region and $F,G\colon  U\to \mathbb{H}$ be functions. Suppose that the limit $p=\lim_{x	\to q} G(x)$ exists for some $q\in\mathbb{H}$. Then:\\
 		(1) If $p\neq 0$ and $\lim_{x	\to {}^{p}q}F(x)$ exists,  the limit $\lim_{x\to q}(F\diamond G)(x)$ exists and 
 		$$\lim_{x\to q}(F\diamond G)(x)= (\lim_{x	\to {}^{p}q} \, F(x))\lim_{x	\to q} G(x).$$
 		(2) If $p=0$ and $F$ is bounded on an open neighborhood of the orbit  $O(q)$, then $\lim_{x\to q}(F\diamond G)(x)=0$.
 	\end{lemma}
 	Next, we show that the skew product of continuous functions is continuous. 
 	
 	\begin{proposition}\label{(P)contiunuityofdiamond}
 		Let $F,G\colon  U\to \mathbb{H}$ be functions on a region. Then:\\
 		(a) If $F$ and $G$ are continuous at each point of an orbit $O\subset U$, then so is $F\diamond G$.\\
 		(b) If $F$ is skew invertible and continuous on $U$, then its skew inverse is continuous on $U$.  
 	\end{proposition} 
 	\begin{proof}
 		(a) It follows from Lemma \ref{(L)limitComposition} and the fact that any orbit in $\mathbb{H}$ is compact.  \\
 		(b) Let $F^{\langle-1\rangle}$ denote the skew inverse of $F$. By Lemma \ref{(L)Invertibility}, $F$ is nonzero on $U$ and the map $g(x)\colon  ={}^{F(x)}x$ is a bijection of $U$ onto itself. Note that $g\colon  U\to U$ is continuous.  By  Brouwer's invariance of domain theorem, $g$ is a homeomrphism, that is, its inverse (with respect to composition) $h\colon  U\to U$ is continuous. Since 
 		$$F^{\langle-1\rangle}(x)=F(h(x))^{-1}, \text{ for all }x\in U,$$ we conclude that $F^{\langle-1\rangle}$ is continuous.   
 	\end{proof}
 	Regarding limits of skew products of sequences of functions, we have the following result.
 	\begin{proposition}\label{(P)continuityofskewproduct}
 		Let $\{F_n\}$ and $\{G_n\}$ be sequences of continuous functions on an invariant compact subset $X$ of $\mathbb{H}$. Suppose that the sequences $\{F_n\}$ and $\{G_n\}$ converge uniformly to functions $F,G\colon  X\to \mathbb{H}$, respectively. If the function $F$  is skew convex on $X$, then the sequence $\{F_n\diamond G_n\}$ converges uniformly to the function $F\diamond G$. 
 	\end{proposition}
 	\begin{proof}
 		Since $F$ is skew-convex, we can write
 		$$F\diamond G-F_n\diamond G_n=(F_n-F)\diamond G_n+F\diamond(G_n-G).$$
 		Note that  $(F_n-F)\diamond G_n$ and $F\diamond(G_n-G)$ converge uniformly to the zero function, from which the desired result follows. 
 	\end{proof}

 \end{subsection}
\begin{subsection}{The skew derivative}
We begin by introducing the definition of a skew-differentiable function.
\begin{definition}\label{(D)skewderivatvie}
	A function $F\colon U\to \mathbb{H}$ on a region $U\subset \mathbb{H}$ is said to be \textit{(left) skew differentiable} at $q\in U$ if there exists a function $\mathcal{D}_qF\colon U\to \mathbb{H}$ which is continuous at  all $p\in O(q)$ and satisfies
	$$F=F(q)+\mathcal{D}_q(F)\diamond (T-q), \text{ as functions on } U.$$
	A function $F\colon U\to \mathbb{H}$ is called (left) skew regular on a region $U$ if it is (left) skew differentiable at all points in $U$. 
\end{definition}
Note that the function $\mathcal{D}_q(F)$ introduced in this definition is unique because it is continuous on $O(q)$ and
$$\mathcal{D}_q(F)=(F-F(q))\diamond  (T-q)^{\langle -1 \rangle}, \text{ on } U\setminus O(q).$$
\begin{remark}
	The function $\mathcal{D}_q(F)$ has also been used in the context of slice-regular functions. For more details, see \cite[Section 2]{Anewseriesexpansion}. 
\end{remark}
  \begin{remark}
	Similarly, one can define the concept of right skew-differentiability  using the right skew product. We leave to the reader the analogous definitions and results concerning right skew-differentiability. Since we mainly work with the concept of left skew-differentiability, we shall often drop the adjective ``left".
\end{remark}
As a direct consequence of Proposition \ref{(P)SkeconvexT-q}, we obtain the following result.  
\begin{proposition}\label{(P)skewdiffimpliesconvex}
	Let $F\colon U\to \mathbb{H}$  be a function on a region. If $F$ is skew differentiable at some point of an orbit $O\subset U$, then $F$ is skew convex on the orbit $O$. In particular, if $F$ is skew regular on $U$, then it is skew convex on $U$.
\end{proposition}
We show that skew regularity implies continuity.
\begin{proposition}\label{(P)Skew-diffimpliescontinuity}
	Let $F\colon U\to \mathbb{H}$ be a function on a region. If $F$ is skew differentiable at $q\in U$, then $F$ is continuous at $q$.
\end{proposition}
\begin{proof}
	  Since the orbit $O(q)$ is compact and $\mathcal{D}_q(F)$ is continuous at all $p\in O(q)$, we see that $\mathcal{D}_q(F)$ is bounded on an open neighborhood of  $O(q)$. Now, the result follows from Part (2) of Lemma \ref{(L)limitComposition}.
 	
\end{proof}

Our definition of the (left) skew derivative is as follows. 
\begin{definition}\label{(D)skewderivatvie2}
	Let $F\colon U\to \mathbb{H}$ be a function on a region $U\subset \mathbb{H}$. Suppose that $F$ is skew differentiable at  $q_0\in U$. Then, the (left) skew derivative of $F$, denoted by $F'(q_0)$, is defined to be the limit
	$$F'(q_0)=\lim_{ O(q_0) \notniFromTxfonts q\to q_0}\left((F-F(q_0))\diamond  (T-q_0)^{\langle-1\rangle}\right) \, (q).$$
\end{definition}
In the particular case where $r\in U$ is real, we have
$$
F'(r)=\lim_{q\to r}\left( F\left( q\right) -F(r)\right) \, (q-r)^{-1}.$$ 
In order to derive some basic properties of the skew derivative, we need the following lemma. 
\begin{lemma}\label{(L)derivativeofske-convex}
	Let a function $F\colon U\to \mathbb{H}$ on a region $U\subset \mathbb{H}$ be skew differentiable at $q\in U$. If $F$ is skew convex on $U$, so is $\mathcal{D}_q(F)$, and  we have
	$$\mathcal{D}_q(F)(p)=F'(q)+\left( F'(q)-F^{o}(q)\right) (q-\overline{q})^{-1}(p-q), \text{ for all } p\in O(q).$$
\end{lemma} 
\begin{proof}
	Note that the function $(T-q)^{\langle-1\rangle}$ is defined (by Proposition \ref{(P)T-c}) and continuous on $U\setminus O(q)$ (by Proposition \ref{(P)contiunuityofdiamond}). In the light of Lemma \ref{(L)limitskewconvexorbit}, the first part follows from the facts that $	\mathcal{D}_q(F)=\left( F-F(q)\right) \diamond (T-q)^{\langle-1\rangle}$ is skew convex on $U\setminus O(q)$, and $\mathcal{D}_q(F)$ is continuous at all $p\in O(q)$. The second part is a consequence of the facts that $\mathcal{D}_q(F)$ is skew convex and \\ $\mathcal{D}_q(F)(q)=F'(q), \, \mathcal{D}_q(F)(\overline{q})=F^{o}(q)$. 
\end{proof}

In the following proposition, some further properties of the skew derivative are summarized.
\begin{proposition}\label{(P)SkewderivativeFirstproperties}
	Let $F, G\colon U\to \mathbb{H}$ be functions on a region $U\subset \mathbb{H}$. \\
	(1)  Let $q\in U$. If $F$ and $G$ are skew differentiable at $q$, then $F+G$ is also skew differentiable at $q$ and $(F+G)'(q)=F'(q)+G'(q)$. \\
	(2) If $F,G$ are skew regular on $U$, then so is $F\diamond G$. Moreover, we have
	$$(F\diamond G)'=F'\diamond G+F\diamond G'.$$ 
	(3) If  $F$ is skew regular and skew invertible on $U$, then its skew inverse is skew regular on $U$, and moreover, we have
	$$
	\left( F^{\langle-1\rangle}\right)'= -F^{\langle-1\rangle}\diamond F'\diamond F^{\langle-1\rangle}.
	$$
\end{proposition}
\begin{proof}
	(1) Trivial.\\
	(2) Fix $q\in U$. We only treat the case $G(q)\neq 0$, and leave the easier case $G(q)=0$ to the reader. Set 
	$$E_1=\mathcal{D}_{{}^{G(q)}q}(F) \text{ and } E_2=\mathcal{D}_q(G).$$ We have
	$$F=F({}^{G(q)}q)+E_1\diamond (T-{}^{G(q)}q) \text{ and } G=G(q)+E_2\diamond (T-q), \text{ on } U.$$
	Since $F$ and $E_1$ are skew convex (see Proposition \ref{(P)skewdiffimpliesconvex} and Lemma \ref{(L)derivativeofske-convex}), we have
	$$F\diamond G=(F\diamond G)(q)+\left( E_1\diamond G(q)+F\diamond E_2\right)\diamond (T-q).$$
	By Lemma \ref{(L)limitComposition}, the function $E_1\diamond G(q)+F\diamond E_2$ is continuous at all $p\in O(q)$. Therefore, $F\diamond G$ is skew regular at $q$ and
	$$\mathcal{D}_q(F\diamond G)= \mathcal{D}_{{}^{G(q)}q}(F)\diamond G(q)+F\diamond \mathcal{D}_q(G).$$
	Evaluating at $q$, we obtain 
	$$(F\diamond G)'(q)=(F'\diamond G)(q)+(F\diamond G')(q).$$
	(3) Fix $p\in U$ and set $q={}^{F^{\langle-1\rangle}(p)}p$. Note that the quaternion $F(q)$ is the inverse of the quaternion $F^{\langle-1\rangle}(p)$. Set $E=\mathcal{D}_q(F).$
	We can write
	\begin{align*}
		F^{\langle-1\rangle}-F^{\langle-1\rangle}(p)=& -F^{\langle-1\rangle}\diamond(F-F(q))\diamond F^{\langle-1\rangle}(p)\\
		=& -F^{\langle-1\rangle}\diamond E\diamond (T-q)\diamond F^{\langle-1\rangle}(p)\\
		=& -F^{\langle-1\rangle}\diamond E\diamond  F^{\langle-1\rangle}(p) \diamond (T-p),
	\end{align*}
	from which the result follows. 
	
\end{proof}
As an application of this proposition, we show that the skew derivative  is skew convex. 
\begin{corollary}
	Let  $F\colon U\to \mathbb{H}$  be a  function on a region. If $F$ is skew regular on $U$, then $F'$ is skew convex on $U$. 
\end{corollary}

\begin{proof}
	Using Proposition \ref{(P)skewdiffimpliesconvex}, we see that $F$ is skew convex on $U$. By definition, we have 
	$$F\diamond(a+b)=F\diamond a+F\diamond b, \text{ for all } a,b\in \mathbb{H}.$$
	Part (2) of Proposition \ref{(P)SkewderivativeFirstproperties} shows that $$F'\diamond(a+b)=F'\diamond a+F'\diamond b, \text{ for all } a,b\in \mathbb{H}.$$
	Therefore, $F'$ is skew convex on $U$.
\end{proof}
Next, we present the chain rule for skew-differentiable functions. 
\begin{proposition}\label{(P)Chainrule}
	Let $\phi\colon U\to U'$ be a continuous action-preserving function between regions in $\mathbb{H}$. If $\phi$ is skew differentiable at  some $q_0\in U$ and  a function $F\colon U'\to\mathbb{H}$ is skew differentiable at $\phi(q_0)$, then $F\circ \phi$ is skew differentiable at $q_0$.   Moreover,   we have
	$$\mathcal{D}_{q_0}(F\circ \phi)=\left(\mathcal{D}_{\phi(q_0)}(F)\circ \phi \right) \diamond \mathcal{D}_{q_0}(\phi).
	$$
	 In particular, $(F\circ \phi)'(q_0)=F'(\phi(q_0))\phi'(q_0)$ if $\phi'(q_0)$ commutes with $\phi(q_0)$. We also have $(F\circ \phi)^{o}(q_0)=F^{o}(\phi(q_0))\phi^{o}(q_0)$. 
\end{proposition}
\begin{proof}
	Set
	 $$E_1=\mathcal{D}_{\phi(q_0)}(F) \text{ and } E_2=\mathcal{D}_{q_0}(\phi).$$
		Then $$F=F(\phi(q_0))+E_1\diamond (T-\phi(q_0)) \text{ on } U',$$ 
		 $$\phi=\phi(q_0)+E_2\diamond (T-q_0) \text{ on } U .$$ 
		Using the fact that $\phi$ is action preserving, one can easily check that
		$$F\circ \phi=(F\circ \phi)(q_0)+\left( (E_1\circ \phi)\diamond E_2\right) \diamond (T-q_0) \text{ on } U .$$
		Since $(E_1\circ \phi)\diamond E_2$ is continuous on $U$, we conclude that $F\circ \phi$ is skew differentiable at $q_0$, and 
		$$\mathcal{D}_{q_0}(F\circ \phi)=\left(\mathcal{D}_{\phi(q_0)}(F)\circ \phi \right) \diamond \mathcal{D}_{q_0}(\phi).
		$$
	The chain rule for the orbital derivative is trivial. 	
\end{proof}
Regarding the inverse of a skew-differentiable function with respect to function composition, we prove the following result. 
\begin{proposition}\label{(P)compositioninverse}
		Let $\phi\colon U\to \mathbb{H}$ be a continuous action-preserving function on a region $U$. Suppose that $\psi:U'\to\mathbb{H}$ is a continuous function on a region $U'$ such that $\phi(U)\subset U'$ and $\psi(\phi(q))=q$ for all $q\in U$. If $\psi$ is skew differentiable at some $\phi(q_0)\in U'$ and the function $\mathcal{D}_{\phi(q_0)}(\psi)$ does not have a root on the orbit $O(\phi(q_0))$, then $\phi$ is skew differentiable at $q_0$. Moreover, we have
		$$\mathcal{D}_{q_0}(\phi)=(\mathcal{D}_{\phi(q_0)}(\psi)\circ \phi)^{\langle -1 \rangle}.$$
		In particular, $\phi'(q_0)=\psi'(\phi(q_0))^{-1}$ if $\phi'(q_0)$ commutes with $\phi(q_0)$.
\end{proposition}
\begin{proof}
	Setting $E=\mathcal{D}_{\phi(q_0)}(\psi)$, we have 
	 $$\psi=q_0+E\diamond (T-\phi(q_0)) \text{ on } U'.$$
	Composing with $\phi$ and using the fact that $\phi$ is action-preserving, we obtain
	$$T=q_0+(E\circ \phi)\diamond (\phi-\phi(q_0)) \text{ on } U .$$
	Since $\phi$ is 1-1 and $E\circ \phi$ does not have a zero on the orbit $O(q_0)$, one can see that $E\circ \phi$ is nowhere zero on $U$. By Proposition \ref{(P)skeweinvertibleonH}, $E\circ \phi$ is skew invertible on $U$. It follows that 
			 	$$\phi=\phi(q_0)+(E\circ \phi)^{\langle -1 \rangle}\diamond (T-q_0) \text{ on } U .$$
	By Part (b) of Proposition \ref{(P)contiunuityofdiamond}, $(E\circ \phi)^{\langle -1 \rangle}$ is continuous. Therefore, $\phi$ is skew differentiable at $q_0$, and 
	$$\mathcal{D}_{q_0}(\phi)=(\mathcal{D}_{\phi(q_0)}(\psi)\circ \phi)^{\langle -1 \rangle}.$$	
\end{proof}
As an application of this proposition, one can show that the principal quaternionic logarithm $Log:\mathbb{H}\setminus (-\infty,0]\to \mathbb{H}$ (see \cite[Definition 3.4]{Weierstrassfactorization}) is skew regular, using the facts that $e^{Log (q)}=q$, for all $q\in \mathbb{H}\setminus (-\infty,0]$, and $e^q$ is a continuous action-preserving function on  $\mathbb{H}$. Moreover, we have $Log'(q)=q^{-1}$ for all $q\in \mathbb{H}\setminus (-\infty,0]$. The details are left to the reader.  

We conclude this part with some remarks.
\begin{remark}
	In \cite{Extensionresults2009}, the regular product of slice-regular functions is first defined for the series centered at $0$, and then extended to slice-regular functions (see \cite[Definition 5.3]{Extensionresults2009}). Proposition 5.12 in the same reference gives a formula for the regular product of regular functions over axially symmetric slice domains. This formula is just the formula for the right skew product (see Definition \ref{(D)rightproduct}).
\end{remark}
\begin{remark}
	We will prove that any slice-regular function over a symmetric slice domain is skew regular (see Theorem \ref{(T)sliceequivalentskew}).  It follows from Proposition \ref{(P)skewdiffimpliesconvex} that  slice-regular functions over  symmetric slice domains are skew-convex. The so-called representation formulas (see for example \cite[Formulas 1.8,1.9,1.10]{Regularfunctionsofaquaternionicvariable2013}) are simply consequences of the fact that slice-regular functions over symmetric slice domains are skew-convex. 
\end{remark}
\begin{remark}
	The notion of orbital derivative, defined in Definition \ref{(D)OrbitalDerivatvie}, is also introduced in \cite{Anewseriesexpansion}, where it is called spherical derivative. In \cite{Sliceregularfunctionsonrealalternativealgebras}, a Leibniz-type formula is given for the orbital product. The formula given in Proposition \ref{(P)PropertiesOrbitalD}  is a new formula. 
\end{remark}
\begin{remark}
	In \cite[Proposition 1.28]{Powerseriesanalyticityquaternions2012}, it is shown that slice-regular functions over a symmetric slice domain form a noncommutative ring under the regular product. The Leibniz formula for the slice derivative is also given in the same reference (see \cite[Proposition 1.40]{Powerseriesanalyticityquaternions2012}). Proposition \ref{(P)SkewderivativeFirstproperties} gives the corresponding results for skew-regular functions. 
\end{remark}
\begin{remark}
	The author has not been able to find any form of the chain rule for slice-regular functions in the literature. The only result in this direction seems to be Lemma 1.32 in \cite{Regularfunctionsofaquaternionicvariable2013} which, roughly speaking, states that the composition of two slice-regular functions is slice-regular if one of them is slice preserving.  
\end{remark}
 \end{subsection}
\begin{subsection}{Skew regularity and slice regularity}\label{(S)Skew regularity and slice regularity}
In this section, we show that skew-regular functions are slice regular, and moreover, the converse holds over symmetric slice domains.  We need to review some facts regarding slice-regular functions. For more details, see \cite{Regularfunctionsofaquaternionicvariable2013} and references therein. Let $\mathbb{S}$ denote the set of all $I\in\mathbb{H}$ such that $I^2=-1$. For $I\in \mathbb{S}$ and a subset $A$ of $\mathbb{H}$, put 
$A_I=A\cap (\mathbb{R}+\mathbb{R}I)$. Note that $\mathbb{R}+\mathbb{R}I$ is a commutative subfield of $\mathbb{H}$. In fact, it is isomorphic to $\mathbb{C}$ via $x+yI\mapsto x+yi$, using which we shall identify  $\mathbb{R}+\mathbb{R}I$ with the complex plane $\mathbb{C}$. 
\begin{definition}\label{(D)sliceregular}
	A function $F\colon U\to \mathbb{H}$ on an open set $U\subset \mathbb{H}$ is called (right) slice regular if 	for each $I\in \mathbb{S}$, the restriction $F_I$ of $F$ to $U_I$ is holomorphic in the sense that it has continuous partial derivatives and 
	$$\frac{1}{2}\left( \frac{\partial F_I}{\partial x}+\frac{\partial F_I}{\partial y}\,I\right)=0, \text{ on } U_I.$$
	If $F$ is slice regular on $U$, its slice derivative $F_s'(q)$ at $q\in U_I$ is defined to be 
	$$F_s'(q)=\frac{1}{2}\left( \frac{\partial F_I}{\partial x}-\frac{\partial F_I}{\partial y}\,I\right).$$ 
\end{definition}
First, we show that every skew-differentiable function is slice regular. 
\begin{proposition}\label{(P)skewimpliesslice}
	Let $F\colon U\to \mathbb{H}$ be a function on a region. If $F$ is skew regular on $U$, then it is slice regular on $U$, and moreover, its skew derivative is equal to its slice derivative. 
\end{proposition}
\begin{proof}
	Fix $I\in\mathbb{S}$. We need to show that the restriction of $F$ to the set $U_I$ satisfies the Cauchy–Riemann equations.  We choose $J\in\mathbb{S}$ such that $J$ is orthogonal to $I$, that is, $IJ=-JI$. Set $z=x+yI$, where $x,y\in \mathbb{R}$. There exist functions $F_1,F_2\colon U_I\to \mathbb{R}+\mathbb{R}I$ such that 
	$$F(z)=F_1(z)+JF_2(z), \text{ for all } z\in U_I.$$
	Let $z_0=x_0+y_0I\in U_I$ be arbitrary. Since $F$ is skew differentiable at $z_0$, we have
	$$F=F(z_0)+\mathcal{D}_{z_0}(F)\diamond (T-z_0), \text{ as functions on } U.$$
	The fact that $\mathbb{R}+\mathbb{R}I$ is a commutative field implies that 
	$$F(z)=F(z_0)+\mathcal{D}_{z_0}(F)(z)(z-z_0), \text{ for all } z\in U_I, $$
	from which we obtain
	$$F_1'(z_0)+JF_2'(z_0)=\mathcal{D}_{z_0}(F)(z_0)=F'(z_0).$$
	In particular, $F_1$ and $F_2$ are holomorphic on $U_I$.  Therefore, $F$ is slice regular. The identity $F'=F'_s$ can easily be verified. 
\end{proof}
In general, a slice-regular function may not be skew regular. In fact, 
there exist slice-regular functions which are not even continuous (see Example 1.10 in \cite{Regularfunctionsofaquaternionicvariable2013}).  To exclude such pathologies, 
one is forced to restrict one's attention to a certain class  of open sets known as slice domains. Recall that a connected open subset $U$ of $\mathbb{H}$ is called a slice domain if $U\cap \mathbb{R}$ is nonempty and the intersection $U\cap (\mathbb{R}+\mathbb{R}q)$ is connected for all $q\in U\setminus \mathbb{R}$. 
It is known that any slice-regular function on a slice domain $U$ can uniquely be extended to a slice-regular function on the symmetric completion of $U$ (see \cite[Theorem 1.24]{Regularfunctionsofaquaternionicvariable2013}). We remark that the symmetric completion of a set $A\subset \mathbb{H}$ is the smallest invariant subset of $\mathbb{H}$ which contains $A$.   It turns out that the concepts of slice regularity and skew regularity coincide over symmetric (meaning invariant) slice domains.  

\begin{theorem}\label{(T)sliceequivalentskew}
	Let  $U\subset\mathbb{H}$ be  a symmetric slice domain.  A function $F\colon U\to \mathbb{H}$ is slice regular on $U$ if and only if it is skew regular on $U$, in which case its skew derivative equals its slice derivative. 
\end{theorem}   
\begin{proof}
	The ``if" direction has been proved in Proposition \ref{(P)skewimpliesslice}. Conversely, let the function $F\colon U\to \mathbb{H}$ be slice regular. Fix $q\in U$. It follows from Theorem 3.17 in \cite{Regularfunctionsofaquaternionicvariable2013} that there exists a slice-regular function $E\colon U\to \mathbb{H}$ such that
	$$F-F(q)=E\diamond (T-q), \text{ as functions on } U.$$
	Since $E$ is slice regular on a symmetric slice domain, $E$ must be continuous on $U$. Therefore, $F$ is skew differentiable at $q$. As the slice derivative is equal to $E(q)$ by the product rule for the slice derivative (see \cite[Proposition 1.40]{Regularfunctionsofaquaternionicvariable2013}), the proof is complete. 
\end{proof}
Next, we show that strongly analytic functions are skew differentiable. Recall that a function is called strongly analytic on a region if at every point $p$ in the region, the function can locally be represented as a convergent series centered at $p$. For more details, see \cite{Powerseriesanalyticityquaternions2012}. 
\begin{proposition}\label{(P)stronglyanalytic}
	Let $F\colon U\to\mathbb{H}$ be strongly analytic on a region $U$. Then, $F$ is skew regular on $U$. 
\end{proposition}
\begin{proof}
	Fix $p\in U$. Since $F$ is strongly analytic at $p$, we can find an open neighborhood $p\in U_0\subset U$, and $q_0,q_1,...\in\mathbb{H}$ such that 
	$$F=\sum_{m=0}^\infty q_m(T-p)^m, \text{ as functions on } U_0,$$
	where the series converges absolutely and uniformly on $U_0$. Set
	$$E=\sum_{m=1}^\infty q_m(T-p)^{m-1}.$$
	Then $E$ is a continuous function on $U_0$ since the convergence is uniform. Using Proposition \ref{(P)continuityofskewproduct}, one can see that 
	$$F=F(p)+E\diamond (T-p), \text{ as functions on } U_0.$$
	It follow that $F$ is skew differentiable at $p$. 
\end{proof}
	

\end{subsection}
\begin{subsection}{Spherical series representation of skew-regular functions}	
	 In this section, we show that skew-regular functions can locally be represented by spherical series. Let us begin with the following result which is an improvement of the fact that any spherical series is slice regular on its region of convergence (see \cite[Proposition 1.11]{Anewseriesexpansion}). 
	\begin{proposition}\label{(P)sphericalanalyticisskewregular}
		Let $S(T)=\sum_{n=0}^\infty q_nP_O(T)^n$ be a formal spherical power series centered at an orbit $O$, and with radius of convergence $R>0$. Then, $S(T)$ is skew regular on $U(O,R)=\{q\in \mathbb{H}\,|\, |P_O(q)|<R \}.$ Moreover, the skew derivative of $S$ is given by the formula
		$$S'(T)=\sum_{n=1}^\infty q_{n}P_O(T)^{n-1}P'_O(T),$$
		where the series is absolutely and uniformly convergent on the compact subsets of $U(O,R)$.
	\end{proposition}
	\begin{proof}
	 We only prove the first assertion. The proof of the second assertion is similar to that of the corresponding result for ordinary power series over real or complex numbers.  If $O$ is a trivial orbit, the result follows from Theorem 16 in \cite{Powerseriesanalyticityquaternions2012} since every skew-regular function is  slice regular by Proposition \ref{(P)skewimpliesslice}. Assume that $O$ is a nontrivial orbit with $P_O(T)=T^2-2x_0T+y_0^2$. 
	 Fix $q_0\in U(O,R)$. For every $n\geq 1$, there exists $P_n(T)\in\mathbb{H}[T]$ such that 
	 $$P_O(T)^n=P_O(q_0)^n+P_n(T) (T-q_0), \text{ as elements of } \mathbb{H}[T].$$
	 It is easy to check that $P_1(T)=T+q_0-2x_0$. Using induction on $n$, we obtain the following recursion formula
	 $$P_{n+1}(T)=P_O(q_0)^n(T+q_0-2x_0)+P_O(T)P_n(T).$$
	 Consider the series $E(T)=\sum_{n=0}^\infty q_nP_n(T)$. I claim that this series is uniformly convergent on the compact subsets of  $U(O,R)$. Let $q\in U(O,R)$ and $m$ be the maximum of the norms $|P_O(q)|$ and  $|P_O(q_0)|$. Using the recursion formula, one can easily show that 
	 $$
	  |P_n(q)|\leq n|q+q_0-2x_0|m^{n-1}, \text{ for all } n\geq 1.
	 $$
	 Now, the claim is easily proved using the Weierstrass $M$-test. In particular, we see that $E$ is continuous on $U(O,R)$. Since 
	 $$S=S(q_0)+E\diamond (T-q_0),$$ the result follows. 
	\end{proof}
	 We need some preparations about line integrals over quaternions. 
	Suppose that a curve $\gamma\colon  [a,b]\to \mathbb{H}$ is  piecewise smooth. For $\mathbb{H}$-valued functions $F,G$, defined and being continuous on the image of $\gamma$, we will use the following line integral along $\gamma$:
	$$\int_{\gamma}F(q)\,dq\, G(q)\colonequals \int_{a}^{b}F(\gamma(t))\gamma'(t)G(\gamma(t))dt.$$
	In the following proposition, we use the skew inverse of $T-p$ with respect to the right skew product (see Remark \ref{(R)rightskewinverse}).  
	\begin{proposition}\label{(P)Cauchyformula}
		Let $F\colon  U\to \mathbb{H}$ be skew regular on a region  and $O\subset U$ be a nontrivial orbit such that $U$ contains the closure of the set 
		$$U(O,R)=\{q\in \mathbb{H}\,|\, |P_O(q)|<R \}, \text{ for some } R>0.$$
		For any $I\in\mathbb{S}$, we have
		$$F(p)=\frac{1}{2\pi}\int_{\gamma}F(q)\,(-Idq)\, (T-p)^{\langle -1 \rangle_r}(q), \text{ for all } p\in U(O,R),$$
		where 	 $\gamma$ is the boundary of the set $U(O,R)\cap (\mathbb{R}+\mathbb{R}I)$. Moreover, we have
		$$
		F^o(p)=\frac{1}{2\pi}\int_{\gamma}F(q)\,(-Idq)\, P_{O(p)}(q)^{-1},\,  \text{ for all } p\in U(O,R).
		$$
	\end{proposition}
	\begin{proof}
		Fix $p'\in U(O,R)$. Then, $ O(p')\cap (\mathbb{R}+\mathbb{R}I)=\{p_0,\overline{p}_0\}$ for some $p_0\in\mathbb{H}$. An argument similar to the proof of Lemma 6.3 in \cite{Regularfunctionsofaquaternionicvariable2013} shows that the desired formula holds for $p=p_0$ and $p=\overline{p}_0$, since $(T-p)^{\langle -1 \rangle_r}(q)=(q-p)^{-1}$ when $pq=qp$. It is easy ot verify that the assignment
		$$p\mapsto \frac{1}{2\pi}\int_\gamma F(q)\,(-Idq)\, (T-p)^{\langle -1 \rangle_r}(q),$$
		defines a (left) skew-convex function on $O(p')$. Since $F$ is skew convex on $O(p')$, it follows immediately that the desired formula holds on $O(p')$. This completes the proof of the first formula. The second formula is derived using the first formula and the definition of the orbital derivative.   
	\end{proof}
Let us give two remarks about this propositions. 
	\begin{remark}
		The first formula given in Proposition \ref{(P)Cauchyformula} can be regarded as an analogue of the Cauchy formula for holomorphic functions.  Although our formula looks similar to Formula 6.4  
		in \cite{Regularfunctionsofaquaternionicvariable2013}, there is a fundamental difference between the two formulas.  Using our notations, we can rewrite the left-hand analogue of Formula 6.4 in \cite{Regularfunctionsofaquaternionicvariable2013} as
		$$F(p)=\frac{1}{2\pi}\int_{\gamma}F(q)\,(-Idq)\, (q-T)^{\langle -1 \rangle}(p).$$
		This formula does not seem suitable for deriving spherical series presentations of skew-regular functions (see Proposition \ref{(P)skewdiffimpliespherical}). 
	\end{remark}
	\begin{remark}
		As a consequence of Proposition \ref{(P)Cauchyformula}, we see that the orbital derivative of a skew-regular function is continuous. In fact, one can use the second formula in the proposition to show that the orbital derivative of a skew-regular function is real analytic. However, it may not be skew regular.  
	\end{remark}	
We also need the following lemma. 
	\begin{lemma}\label{(L)SphericalT-pinverse}
		Let $p\in \mathbb{H}$, and $O\subset \mathbb{H}$ be a nontrivial orbit such that $p\notin O$. Then, the right skew-regular function $(T-p)^{\langle -1 \rangle_r}$ can be represented as the (right) series 
		$$(T-p)^{\langle -1 \rangle_r}=\sum_{n=0}^\infty P_O(T)^{-n-1}(T+p-2x_0)P_O(p)^n,$$
		where the series is absolutely and uniformly convergent on the compact subsets of the region $$\{q\in \mathbb{H}|\, |P_O(p)|<|P_O(q)|\}.$$
	\end{lemma}
\begin{proof}
	We have 
	 $$P_O(T)=P_O(p)+(T-p)\diamond_r (T+p-2x_0), \text{ as functions on } \mathbb{H}.$$
	Note that $P_O(T)\in \mathbb{R}[T]$ is skew invertible on $\mathbb{H}\setminus O(p)$ with respect to the left and right skew products (see the discussion at the end of Section \ref{(S)Skewinvertible}).  
	 It follows that $$1-P_O(p)\diamond_r P_O(T)^{-1}=(T-p)\diamond_r (T+p-2x_0)\diamond_r P_O(T)^{-1},$$ 
	 as functions on $\mathbb{H}\setminus O(p)$. Then, one can see that
	$$
	(T-p)^{\langle -1\rangle_r}=(T+p-2x_0)\diamond_r P_O(T)^{-1}\diamond_r \left( 1-P_O(p)\diamond_r P_O(T)^{-1}\right)^{\langle -1\rangle_r},
	$$
	as functions on $\{q\in \mathbb{H}|\, |P_O(p)|<|P_O(q)|\}.$ Now, 
	the result follows from the identity 
	$$
\left( 1-P_O(p)\diamond_r P_O(T)^{-1}\right)^{\langle -1\rangle_r}=\sum_{n=0}^\infty P_O(T)^{-n}P_O(p)^n,
	$$
	where the (right) series converges absolutely and uniformly on the compact subsets of the region  $\{q\in \mathbb{H}|\, |P_O(p)|<|P_O(q)|\}.$
\end{proof}
Now, we can prove that skew-regular functions can locally be presented as spherical series.  
	\begin{proposition}\label{(P)skewdiffimpliespherical}
		Let $F\colon  U\to \mathbb{H}$ be skew regular on a region $U$. Let $O\subset U$ be a nontrivial orbit such that $U$ the closure of the set 
		$$U(O,R)=\{q\in \mathbb{H}\,|\, |P_O(q)|<R \}.$$  Then, there exist spherical series $S_1(T), S_2(T)$, centered at $O$, which are convergent on $U(O,R)$ and satisfy 
		$$F(T)=S_1(T)+S_2(T)\diamond T, \text{ as functions on } U(O,R).$$ 
	\end{proposition}
	\begin{proof}
	Fix $I\in\mathbb{S}$ and let $\gamma$ be the boundary of $U(O,R)\cap (\mathbb{R}+\mathbb{R}I)$. An application of Proposition \ref{(P)Cauchyformula} gives
		$$F(p)=\frac{1}{2\pi}\int_{\gamma}F(q)\,(-Idq)\, (T-p)^{\langle -1 \rangle_r}(q), \text{ for all } p\in U(O,R).$$ 
		Note that $|P_O(p)|<|P_O(q)|=R$ for all $p\in U(O,R)$ and $q\in \gamma$. Let $P_O(T)=T^2-2x_0T+y_0^2$.
		Using Lemma \ref{(L)SphericalT-pinverse}, we see that
	\begin{align*}
		F(p)=&\frac{1}{2\pi}\int_{\gamma} F(q)\,(-Idq)\, (T-p)^{\langle -1 \rangle_r}(q)\\
		=&\frac{1}{2\pi}\int_{\gamma} F(q)\,(-Idq)\, \left( \sum_{n=0}^\infty P_O(q)^{-n-1}(q+p-2x_0)P_O(p)^n \right)\\
		=&  \sum_{n=0}^\infty\left(\frac{1}{2\pi}\int_{\gamma} F(q)\,(-Idq)\,  P_O(q)^{-n-1}(q+p-2x_0)\right)P_O(p)^n. 
	\end{align*}
	It is allowed to change the order of summation and integration since the (right) series converges absolutely and uniformly on the compact set $\gamma$. 
	Setting 
	$$
	S_1(T)=\sum_{n=0}^\infty\left(\frac{1}{2\pi}\int_{\gamma} F(q)\,(-Idq)\,  P_O(q)^{-n-1}(q-2x_0)\right)P_O(T)^n, \text{ and }
	$$
	$$
	S_2(T)=\sum_{n=0}^\infty\left(\frac{1}{2\pi}\int_{\gamma} F(q)\,(-Idq)\,  P_O(q)^{-n-1}\right)P_O(T)^n,
	$$
	we see that 
	$$F(T)=S_1(T)+S_2(T)\diamond T, \text{ as functions on } U(O,R),$$
	completing the proof. 		
	\end{proof}
	Finally, we state a useful characterization of skew-regular functions. Following \cite[Definition 4.2]{Anewseriesexpansion}, we first introduce a definition.
	\begin{definition}\label{(D)sphericallyanalytic}
		A function $F\colon  U\to\mathbb{H}$ on a region $U$ is called symmetrically analytic if it has a spherical series representation at every orbit $O\subset U$, that is, \\
		(1) If $O\subset U$ is trivial,  there exists a spherical series $S(T)$, centered at $O$, which is convergent in a neighborhood $V\subset U$ of $O$ and satisfy \\
		$F(T)=S(T), \text{ as functions on } V.$\\
		(2) If $O\subset U$ is nontrivial, there exist spherical series $S_1(T),S_2(T)$, centered at $O$, which are convergent in a neighborhood $V\subset U$ of $O$ and satisfy
		$$F(T)=S_1(T)+S_2(T)\diamond T, \text{ as functions on } V.$$
	\end{definition} 
	Now we come to the main result of this section. 
	\begin{theorem}\label{(T)skewregularequivalentspherical}
		A function $F\colon  U\to \mathbb{H}$ on a region $U\subset \mathbb{H}$ is skew regular if and only if it is symmetrically analytic.
	\end{theorem}
	\begin{proof}
		The ``if" direction follows from Proposition \ref{(P)sphericalanalyticisskewregular}. The other direction follows from Proposition \ref{(P)skewdiffimpliespherical} if $O$ is a nontrivial orbit. In the case of trivial orbits, the result follows from Proposition 2.7  in \cite{Extensionresults2009}. 
	\end{proof}
	
	Theorem \ref{(T)skewregularequivalentspherical} has a number of important consequences which will be studied in a sequel to this paper. Let us just mention two direct consequences. The first corollary  is the generalization of the fact that the slice derivative of a slice-regular function is slice regular. 
	\begin{corollary}\label{(C)skewregularininiftlydiff}
		If  a function $F\colon  U\to \mathbb{H}$ is skew regular on a region $U\subset \mathbb{H}$, then it is infinitely skew differentiable. In particular, the  skew derivative $F'\colon  U\to \mathbb{H}$ of $F$  is skew regular on $U$. 
	\end{corollary}
	\begin{proof}
		This result follows from 	Theorem \ref{(T)skewregularequivalentspherical}  and the fact that spherical series can be differentiated term by term. The details are left to the reader. 
	\end{proof}
	In the following corollary, $(T-p)^{\langle -n-1 \rangle_r}$ denotes the $(n+1)$-th power of $(T-p)^{\langle -1 \rangle_r}$ with respect to the right skew product. 
		\begin{corollary}\label{(C)CauchyFormuladerivative}
			Let $F\colon  U\to \mathbb{H}$ be skew regular on a region  and $O\subset U$ be a nontrivial orbit such that $U$ contains the closure of the set 
		$$U(O,R)=\{q\in \mathbb{H}\,|\, |P_O(q)|<R \}.$$
		For any $I\in\mathbb{S}$, we have
		$$F^{(n)}(p)=\frac{n!}{2\pi}\int_{\gamma}F(q)\,(-Idq)\, (T-p)^{\langle -n-1 \rangle_r}(q), \text{ for all } p\in U(O,R),$$
		where 	 $\gamma$ is the boundary of the set $U(O,R)\cap (\mathbb{R}+\mathbb{R}I)$. 
	\end{corollary}
	\begin{proof}
		In view of Corollary \ref{(C)skewregularininiftlydiff}, one can use an argument similar to the proof of Proposition \ref{(P)Cauchyformula} to derive the desired formula. 
	\end{proof}

\end{subsection}
	

\end{section}

\bibliographystyle{plain}
\bibliography{SAbiblan}

 \end{document}